\documentclass[10pt, article]{amsart}
\usepackage{tikz}
\usetikzlibrary{calc}
\usepackage{ae} 
\usepackage[T1]{fontenc}
\usepackage[cp1250]{inputenc}
\usepackage{amsmath}
\usepackage{amssymb, amsfonts,amscd,verbatim}

\usepackage[normalem]{ulem}
\usepackage{hyperref}
\usepackage{indentfirst}
\usepackage{latexsym}
\usepackage{mathrsfs} 
\input xy
\xyoption{all}

\usepackage{amsmath}    

\theoremstyle{plain}
\newtheorem{Pocz}{Poczatek}[section]
\newtheorem{Proposition}[Pocz]{Proposition}
\newtheorem{Theorem}[Pocz]{Theorem}
\newtheorem{Corollary}[Pocz]{Corollary}

\newtheorem{Lemma}[Pocz]{Lemma}
\newtheorem{Observation}[Pocz]{Observation}
\newtheorem{Notation}[Pocz]{Notation}
\newtheorem{Question}[Pocz]{Question}

\theoremstyle{definition}
\newtheorem{Definition}[Pocz]{Definition}

\theoremstyle{remark}
\newtheorem{Remark}[Pocz]{Remark}

\def\UU{{\mathcal U}}
\def\VV{{\mathcal V}}

\def\dim{\mathrm{dim}}

\numberwithin{equation}{section}
\title[Decomposition complexity with respect to coarse properties]
{Decomposition complexity with respect to coarse properties}

\author{Jerzy Dydak}

\address{University of Tennessee, Knoxville, TN 37996}
\email{jdydak\@@utk.edu}

\date{ \today
}


\begin{document}
\maketitle
\begin{center}
\today
\end{center}

\begin{abstract}
We formalize the concept of a family of metric spaces satisfying a coarse property uniformly and we generalize finite decomposition complexity of Erik Guentner, Romain Tessera, and Guoliang Yu. Of particular interest are results determining sufficient conditions for a metric space to satisfy Property A of Guoliang Yu.
\end{abstract}

\section{Introduction}

Asymptotic dimension was introduced by Gromov for the purpose of studying groups using geometric methods (see \cite{Roe}). In Ostrand (\cite{O$_1$} or
\cite{OstrandDimofMetriSpacesHilbert}) formulation (see \cite{BDLM}) it can be defined as follows:

\begin{Definition}
Suppose $X$ is a metric space. The asymptotic dimension of $X$ is at most $n$
if, for every real number $R > 0$, there is a decomposition of $X$ into a union of
its subsets $X_0, \ldots, X_n$ such that each $X_i$ is the union of a uniformly bounded
and $R$-disjoint family $\mathcal{U}_i$. That means there is a real number $S > 0$ with each member
of $\mathcal{U}_i$ being of diameter at most $S$ and the distance between points belonging
to different elements of $\mathcal{U}_i$ is at least $R$.
\end{Definition}

The concept of asymptotic dimension was generalized by Dranishnikov \cite{Dra1} to that of
$X$ having asymptotic Property C:

\begin{Definition}
Suppose $X$ is a metric space. $X$ has the asymptotic property C
if, for every sequence of real numbers $R_i > 0$, there is a decomposition of $X$ into a finite union of
its subsets $X_0, \ldots, X_n$ for some natural $n$ such that each $X_i$ is the union of a uniformly bounded
and $R_i$-disjoint family $\mathcal{U}_i$. 
\end{Definition}

Yamauchi \cite{Yam} proved that the infinite direct product of integers has asymptotic property C.
See \cite{BG}, \cite{BN}, and \cite{DZ} for recent results concerning asymptotic property C.

The next generalization of asymptotic dimension appeared in \cite{GTY1} under the name of \textbf{finite
decomposition complexity} (FDC) and was introduced to study questions concerning the topological rigidity of manifolds \cite{GTY1} \cite{GTY2}. 
FDC naturally arises in the following contexts at the border of large-scale geometry and topology:\\
1. The Bounded Borel Conjecture which asks the following: Is a quasi-isometry between uniformly contractible Riemannian manifolds necessarily a bounded distance from a homeomorphism? In dimensions higher than four, surgery theory reduces this problem to proving the bounded Farrell-Jones Isomorphism Conjecture (a coarse geometric analogue of the usual Farrell-Jones Conjecture) which asserts that a certain assembly map in bounded L-theory is an isomorphism. FDC was defined for the purpose of developing a large scale cutting and pasting method to attack these conjectures with the help of the controlled Mayer-Vietoris sequence of Ranicki-Yamasaki \cite{RY1, RY2}. Using these techniques it was proved in \cite{GTY1} that if the fundamental group of a closed aspherical manifold has finite decomposition complexity, then its universal cover is boundedly rigid, that is, satisfies the Bounded Borel Conjecture, and the manifold itself is stably rigid.\\
2. The integral Novikov conjecture for the algebraic K-theory of group rings $R[\Gamma]$.
If $ \Gamma$ has FDC, it was studied in \cite{RTY}. For a discrete group $ \Gamma$, the classical Novikov conjecture on the homotopy invariance of higher signatures is implied by rational injectivity of the Baum-Connes assembly map \cite{BC}. In Yu \cite{GYu} and Skandalis-Tu-Yu \cite{STG}, injectivity of the Baum-Connes map was proved for groups coarsely embeddable into Hilbert space. Using this result, Guentner, Higson, and Weinberger \cite{GHW} proved the Novikov conjecture for linear groups. This was followed by the work of Guentner, Tessera, and Yu \cite{GTY1}, who proved the integral Novikov conjecture (establishing integral injectivity of the L-theoretic assembly map) for geometrically finite FDC groups (i.e. those with a finite CW model for their classifying space), and hence the stable Borel Conjecture for closed aspherical manifolds whose fundamental groups have FDC.

The class of metric spaces with finite decomposition complexity contains all countable linear groups equipped with a proper (left-)invariant metric (see \cite{GTY1} \cite{GTY2}).

Notice that Dranishnikov and Zarichnyi \cite{DZ} introduced a simpler concept, namely \textbf{straight finite decomposition complexity} which is much closer in spirit to the asymptotic property C.
That concept was subsequently generalized in \cite{Dyd1} and, independently, in \cite{RR}.
The common feature of all the generalizations is that they imply Property A of G.Yu
(see \cite{NowakYu}, \cite{CDV1}, \cite{CDV2}, and \cite{CDV3}  for various characterizations of it).

At a conference in Regensburg (July 2016) Daniel Kasprowski considered a generalization of finite decomposition complexity that involved arriving at spaces satisfying Property A uniformly and asked if such spaces have Property A. In this paper we introduce a new kind of decomposition complexity which implies Property A thus answering Kasprowski's question positively.

In contrast to most papers on coarse geometry, we do not restrict ourselves to metric spaces only. It is more convenient to consider a wider class of spaces.

\begin{Definition}
An \textbf{$\infty$-pseudo-metric space} $X$ is a set with a distance function $d$ that satisfies weaker axioms that a metric:\\
1. $d(x,x)=0$ and $d(x,y)=d(y,x)\ge 0$ for all $x,y\in X$,\\
2. $d(x,y)$ is allowed to assume the value of $\infty$,\\
3. $d(x,y)\leq d(x,z)+d(z,y)$ if $d(x,z), d(z,y) < \infty$.
\end{Definition}

\begin{Remark}
The author is grateful to Pawe\l\ Grzegrz\' olka for useful comments on the paper.
\end{Remark}

\section{Trees of partitions of unity}
In this section we introduce the concept of a tree of partitions of unity which will be our main tool in investigating decomposition complexity of metric spaces.

A \textbf{partition of unity} $\phi$ on a set $X$ is a family
of functions $\phi_s:X\to [0,1]$, $s\in S$, such that $\sum\limits_{s\in S}\phi_s(x)=1$ for all $x\in X$. 
We can consider $\phi$ to be a function from $X$ to $\Delta(S)\subset l_1(S)$, the full simplicial complex spanned by $S$.

 $\phi$ is \textbf{trivial} if its index set consists of one point.

A \textbf{stratum} of a partition of unity $\phi$ on a set $X$ is any of the sets $\phi_s^{-1}(0,1]$,
$s\in S$. Alternatively, it is
 the point-inverse $\phi^{-1}(st(s))$ of some star of a vertex $s$ in the range of $\phi$.

A \textbf{point-finite partition of unity} $\phi$ is one for which the set $\{s\in S | \phi_s(x) > 0\}$ is finite for each $x\in X$. In other words, each point of $X$ belongs to finitely many strata of $\phi$. 

The \textbf{dimension} of $\phi$ is at most $n$ if each point of $X$ belongs to at most $(n+1)$ strata of $\phi$. In other words, the \textbf{multiplicity} of $\phi$ (or its strata), is at most $(n+1)$. 

A \textbf{partial partition of unity} on $X$ is a partition of unity on a subset of $X$.

\begin{Notation}
The notation $\phi\leq \psi$ means that $\phi$ is a partial partition of unity on a stratum of $\psi$.

\end{Notation}

\begin{Definition}\label{probability tree}
By a \textbf{probability tree} we mean a rooted tree $\mathcal{T}$ with each edge assigned a non-negative number such that the following conditions are satisfied:\\
1. For each vertex $v$ of $\mathcal{T}$ that is not a leaf, the sum of values on edges stemming from $v$ to vertices one depth higher is $1$,\\
2. The sum of probabilities of all leaves of $\mathcal{T}$ is $1$, where each vertex $v$ that is not a root is given the probability equal to the product of all values on edges belonging to the unique path leading from the root to $v$.
\end{Definition}

\begin{Definition}\label{tree of partitions of unity}
A \textbf{tree of partitions of unity} $\mathcal{T}$ on a set $X$ is a rooted tree whose vertices are partial partitions of unity on $X$ satisfying the following conditions:\\
1. The partition of unity at the root of $\mathcal{T}$ is a full partition of unity, i.e. its domain is the whole $X$,\\
2. If there is an edge from $\psi$ at depth $n$ to $\phi$ at depth $n+1$, then $\phi\leq \psi$, i.e. $\phi$ is a partition of unity on a stratum of $\psi$ \\
3. Each $\psi$ at depth $n$ which is not a leaf of $\mathcal{T}$ is indexed
by all partitions of unity at depth $n+1$ which have an edge to $\psi$,\\
4. The leaves of $\mathcal{T}$ are trivial partitions of unity,\\
5. For each $x\in X$, the restriction $\mathcal{T}|\{x\}$ of $\mathcal{T}$ to $\{x\}$ is a probability tree
if each edge $[\psi,\phi]$ is given the value $\psi_{\phi}(x)$.
\end{Definition}

\begin{Observation}
Every tree of partitions of unity $\mathcal{T}$ on a set $X$ induces a partition of unity $\phi(\mathcal{T})$ on $X$ that is indexed by leaves of $\mathcal{T}$ and whose strata are exactly the leaves of $\mathcal{T}$. It is defined as follows: given $x\in X$ and given a leaf $L$ of $\mathcal{T}$, $\phi(\mathcal{T})_L(x)$ is the product of 
$\psi_v(x)$, with $\psi$ and $v$ being on the geodesic joining $L$ to the root of $\mathcal{T}$,
and with $v$ being at height 1 higher than the height of $\psi$.
\end{Observation}

\begin{Definition}
Given $R, \epsilon > 0$, a function $\phi:X\to Y$ between $\infty$-pseudo-metric spaces is 
\textbf{$(\epsilon,R)$-continuous} if $d_X(x,y)\leq R$ implies $d_Y(\phi(x),\phi(y))\leq \epsilon$.
\end{Definition}

\begin{Lemma}\label{FundLemmaOnRootedTreesOfPUs}
Suppose $\phi=\{\phi_s\}_{s\in S}$ is a partition of unity on an $\infty$-pseudo-metric space $X$ and $R, \epsilon_0, \epsilon_1 > 0$. If $\phi$ is $(\epsilon_0,R)$-continuous and, for each $s\in S$, there is 
a function $f_s$ from $\phi_s^{-1}(0,1]$ to the unit sphere of $l_1(C_s)$, with $\{C_s\}_{s\in S}$ being mutually disjoint, that is $(\epsilon_1,R)$-continuous, then the function $f$ from $X$ to the unit sphere of $l_1(\bigcup\limits_{s\in S} C_s)$
defined by 
$$f(x)=\sum\limits_{s\in S} \phi_s(x)\cdot f_s(x)$$
 is $(\epsilon,R)$-continuous, where
$$\epsilon=\epsilon_0+\epsilon_1.$$
\end{Lemma}
\begin{proof}
Suppose $d(x,y)\leq R$ for some $x,y\in X$. If both $x$ and $y$ belong to the stratum of $\phi$ induced by $s\in S$, then
$$| \phi_s(x)\cdot f_s(x)-  \phi_s(y)\cdot f_s(y)|=|(\phi_s(x)-\phi_s(y))\cdot f_s(x)+\phi_s(y)\cdot (f_s(x)-f_s(y)|\leq$$
$$ |\phi_s(x)-\phi_s(y)|+\phi_s(y)\cdot |f_s(x)-f_s(y)|\leq  |\phi_s(x)-\phi_s(y)|+\phi_s(y)\cdot \epsilon_1.$$
If only one of them, say $x$, belongs to the stratum of $\phi$ induced by $s\in S$, then
$$| \phi_s(x)\cdot f_s(x)-  \phi_s(y)\cdot f_s(y)|=|\phi_s(x)\cdot f_s(x)|=|\phi_s(x)-\phi_s(y)|.$$
The same equality holds if both $x$ and $y$ do not belong to the stratum of $\phi$ induced by $s\in S$. In all cases we can state that 
$$| \phi_s(x)\cdot f_s(x)-  \phi_s(y)\cdot f_s(y)|\leq |\phi_s(x)-\phi_s(y)|+\phi_s(y)\cdot \epsilon_1.$$
Therefore
$$|f(x)-f(y)|=\sum\limits_{s\in S} | \phi_s(x)\cdot f_s(x)-  \phi_s(y)\cdot f_s(y)|\leq  $$
$$\sum\limits_{s\in S}(|\phi_s(x)-\phi_s(y)|+\phi_s(y)\cdot \epsilon_1) \leq \epsilon_0+\epsilon_1.$$
\end{proof}

\begin{Theorem}\label{FundThmOnRootsOfPUs}
Suppose $\mathcal{T}$ is a tree of partitions of unity on an $\infty$-pseudo-metric space $X$ and $R, \epsilon_n > 0$,
for $n\ge 0$. If each vertex of $\mathcal{T}$ at depth $n$ is $(\epsilon_n,R)$-continuous, then the partition of unity on $X$ induced by $\mathcal{T}$ is $(\epsilon,R)$-continuous, where
$$\epsilon=\sum\limits_{i=0}^\infty \epsilon_i.$$
\end{Theorem}
\begin{proof}
First, consider the case of $\mathcal{T}$ having a finite height $n$. If $n=1$, then the statement is obviously true. Suppose \ref{FundThmOnRootsOfPUs} is valid for all rooted trees of partitions of unity of height at most $k\ge 1$. Consider a rooted tree $\mathcal{T}$  of partitions of unity on $X$ of height $k+1$. Each vertex of $\mathcal{T}$ at height $1$ is a root of a subtree of $\mathcal{T}$
that induces a partition of unity on a stratum of $\phi(\mathcal{T})$ that is $(\epsilon',R)$-continuous, where
$\epsilon'=\sum\limits_{i=1}^\infty \epsilon_i.$ Create a new rooted tree of partitions of unity
$\mathcal{T}'$ on $X$ with the same root as $\mathcal{T}$ but with new vertices at height $1$, namely
the above induced partitions of unity by rooted subtrees. Applying \ref{FundLemmaOnRootedTreesOfPUs} one gets that $\mathcal{T}'$ induces a partition of unity
on $X$ that is $(\epsilon,R)$-continuous, where
$$\epsilon=\sum\limits_{i=0}^\infty \epsilon_i.$$
Notice that partition of unity is identical with the one induced by $\mathcal{T}$.

Suppose $\mathcal{T}$ is an arbitrary tree of partitions of unity on $X$ such that each vertex of $\mathcal{T}$ at depth $n$ is $(\epsilon_n,R)$-continuous. Assume $d(x,y)\leq R$ but
$\phi(x)-\phi(y)| > \epsilon+3\delta$ for some $\delta > 0$.
Without loss of generality assume $X=\{x,y\}$.
Choose finitely many leaves $L$ of $\mathcal{T}$ that contribute at least $1-\delta$ to the probability distribution of
both $x$ and $y$. Let $n$ be the maximum height of those leaves. For each vertex $v$ at height $n$
that is not a leaf, replace the partition of unity at $v$ by the trivial partition of unity.
The partition of unity $\psi$ induced by the new rooted tree is
$(\epsilon,R)$-continuous. On leaves in $L$ both $\phi$ and $\psi$ agree.
Therefore $\phi(x)-\phi(y)| \leq \epsilon+2\delta$, a contradiction.
\end{proof}

\begin{Definition}
Suppose $X$ is an $\infty$-pseudo-metric space and $R > 0$.
 Given $x\in X$, and given $V\subset X$
 we define the \textbf{index} $i_R(x,V)$ of $x$ in $V$
 as the smallest integer $k\ge 0$ such that there is a chain of points $x_0=x, x_1,\ldots, x_k$
 with $x_k\notin V$ and $d(x_i,x_{i+1})\leq R$ for all $i < k$.
 If such a chain does not exist, we put $i_R(x,V)=\infty$.
\end{Definition}

If the multiplicity function $m_{\VV}$ of a cover $\VV=\{V_s\}_{s\in S}$ is finite at each point, then $\VV$ has a natural partition of unity $\phi_R^{\VV}$ associated to it:
$$(\phi_R^{\VV})_s(x)=\frac{i_R(x,V_s)}{\sum\limits_{t\in S}i_{\UU}(x,V_t)}.$$
In case there are indices $t\in S$ such that $i_R(x,V_t)=\infty$, we count the number of such indices, say
there is $k$ of them, and we put
$(\phi_R^{\VV})_s(x)=1/k$ if $i_R(x,V_s)=\infty$ and $(\phi_R^{\VV})_s(x)=0$ if $i_R(x,V_s)< \infty$.

\begin{Lemma}\label{ProjectionOnUnitSphere}
Suppose $C$ is a set of cardinality $2m$ and $x, y\in l_1(C)$ have non-negative coordinates.
If $|x_c-y_c|\leq 1$ for each $c\in C$ and $|x| \ge L$, then
$$\|\frac{x}{|x|}-\frac{y}{|y|}\|\leq \frac{4m}{L}.$$
\end{Lemma}
\begin{proof}
$x_c\cdot |y|-y_c\cdot |x|=(x_c-y_c)\cdot |y|+y_c\cdot (|y|-|x|)$ for each $c\in C$. Therefore
$$|x_c\cdot |y|-y_c\cdot |x||\leq |y|+y_c\cdot 2m$$
and, summing up over all elements of $C$,
$$|x\cdot |y|-y\cdot |x||\leq 2m\cdot |y|+|y|\cdot 2m=4m\cdot |y|.$$
Finally,
$$\|\frac{x}{|x|}-\frac{y}{|y|}\|\leq \frac{4m}{|x|}\leq \frac{4m}{L}.$$
\end{proof}

\begin{Corollary}\label{ExistenceOfREpsilonPUs}
Suppose $X$ is an $\infty$-pseudo-metric space and $m, R, \epsilon > 0$. If $\VV=\{V_s\}_{s\in S}$ is a cover of $X$ of multiplicity at most $m$ at each point and of Lebesgue number at least $\frac{4m\cdot R}{\epsilon}$, then the partition of unity
$\phi_R^{\VV}$ is $(\epsilon, R)$-continuous.
\end{Corollary}
\begin{proof}
Suppose $d(x,y)\leq R$. $i_R(x,V_t)=\infty$ if and only $i_R(y,V_t)=\infty$ for any $t\in S$,
so $\phi_R^{\VV}(x)= \phi_R^{\VV}(y)$ in that case. Therefore assume $i_R(x,V_t) < \infty$
for all $t\in S$. In that case $|i_R(x,V_t)-i_R(y,V_t)|\leq 1$ for all $t\in S$.
Notice that there is a subset $C$ of $S$ of cardinality at most $2m$ such that $i_R(x,V_t)=i_R(y,V_t)=0$ for $t$ outside of $C$. At the same time the sum of $i_R(x,V_t)$, $t\in C$, is at least
$L:= \frac{4m}{\epsilon}$. Applying \ref{ProjectionOnUnitSphere}, one gets that $\phi_R^{\VV}$ is $(\epsilon, R)$-continuous.
\end{proof}

\section{Families satisfying a coarse property uniformly}

Given a coarse property $\mathscr{P}$, it is of interest to ponder the meaning of a family
$\{X_s\}_{s\in S}$ of $\infty$-pseudo-metric spaces to have $\mathscr{P}$ \textbf{uniformly}.
The way it is done in \cite{GuentnerPermanence} amounts to choosing a definition of $\mathscr{P}$ using some parameters and considering $\{X_s\}_{s\in S}$ to have $\mathscr{P}$ uniformly if the same parameters apply to all elements of $\{X_s\}_{s\in S}$ simultanously. However, it is not clear that, if we choose an equivalent definition of $\mathscr{P}$ using different parameters, the two versions of having $\mathscr{P}$ uniformly are equivalent. Therefore we propose a new way of defining the concept. To accomplish it, first we need to define wedges of pointed $\infty$-pseudo-metric spaces.

\begin{Definition}
Suppose $(A_s,x_s)_{s\in S}$ is a family of pointed $\infty$-pseudo-metric spaces.
By the \textbf{wedge} $\bigvee\limits_{s\in S} (A_s,x_s)$ of $(A_s,x_s)_{s\in S}$
we mean the subset of the cartesian product $\prod\limits_{s\in S} (A_s,x_s)$
consisting of all points whose coordinates are equal to $x_s$ for all $s\in S$ but possibly one. It has the natural base point $\{x_s\}_{s\in S}$ and the distance between two points
$\{a_s\}_{s\in S}$ and $\{b_s\}_{s\in S}$ of $\bigvee\limits_{s\in S} (A_s,x_s)$ is defined by
$$\sum\limits_{s\in S}d_s(a_s,b_s).$$
\end{Definition}

\begin{Definition}\label{RegularCoarseProperty}
A coarse property $\mathscr{P}$ is \textbf{regular} if it satisfies the following conditions:\\
1. If an $\infty$-pseudo-metric space $X$ satisfies $\mathscr{P}$ and $X_s$ is an isometric copy of $X$ for $s\in S$, then $\bigvee\limits_{s\in S} (X_s,x_s)$ has $\mathscr{P}$ for any choice of $x_s\in A_s$.\\
2. If each element of a finite family $\{X_s\}_{s\in S}$ of $\infty$-pseudo-metric spaces satisfies $\mathscr{P}$, then $\bigvee\limits_{s\in S} (X_s,x_s)$ has $\mathscr{P}$ for any choice of $x_s\in X_s$.
\end{Definition}

\begin{Corollary}
If a hereditary coarse property $\mathscr{P}$ has union permanence in the sense of Guentner  \cite{GuentnerPermanence}, then it is regular.
\end{Corollary}
\begin{proof}
Suppose $X$ satisfies $\mathscr{P}$ and $A_s\subset X$ for $s\in S$. Consider $Y=\bigvee\limits_{s\in S} (A_s,x_s)$ for some choice of $x_s\in A_s$ and suppose $r > 0$.
Notice $Y\setminus B(y_0,r)$, $y_0$ being the basepoint of $Y$, splits into the union
of an $r$-disjoint family $A_s\setminus B(x_s,r)$ which, together with $B(y_0,r)$,
does have $\mathscr{P}$ uniformly in the sense of Guentner.
Part of union permanence for $\mathscr{P}$ is that $Y$ must have $\mathscr{P}$.
\end{proof}

\begin{Proposition}
Suppose $\mathscr{P}$ is a hereditary regular coarse property and $\{X_s\}_{s\in S}$ is a family of $\infty$-pseudo-metric spaces. If $\bigvee\limits_{s\in S} (X_s,x_s)$ has $\mathscr{P}$ for some choice of basepoints $x_s\in X_s$, then
$\bigvee\limits_{s\in S} (X_s,y_s)$ has $\mathscr{P}$ for every choice of basepoints $y_s\in X_s$.
\end{Proposition}
\begin{proof}
Put $Y=\bigvee\limits_{s\in S} (X_s,x_s)$ and consider $y_s\in Y$ to be the point corresponding to
$y_s\in X_s$. Notice $\bigvee\limits_{s\in S} (Y,y_s)$ contains an isometric copy
of $\bigvee\limits_{s\in S} (X_s,y_s)$, hence $\bigvee\limits_{s\in S} (X_s,y_s)$
has $\mathscr{P}$.
\end{proof}

\begin{Definition}
Suppose $\mathscr{P}$ is a regular coarse property and $\{X_s\}_{s\in S}$ is a family of $\infty$-pseudo-metric spaces. We say that $\{X_s\}_{s\in S}$ 
has $\mathscr{P}$ \textbf{uniformly} if $\bigvee\limits_{s\in S} (X_s,x_s)$ has $\mathscr{P}$
for any choice of basepoints $x_s\in X_s$.
\end{Definition}

\begin{Observation} One can extend the above definition to a class of $\infty$-pseudo-metric spaces as follows:\\
Suppose $\mathscr{P}$ is a regular coarse property. A class $\mathcal{C}$ of $\infty$-pseudo-metric spaces satisfies $\mathscr{P}$ uniformly if for any family $\{X_s\}_{s\in S}$
in $\mathcal{C}$,  \textbf{uniformly} if $\bigvee\limits_{s\in S} (X_s,x_s)$ has $\mathscr{P}$
for any choice of basepoints $x_s\in X_s$.
\end{Observation}

\begin{Observation}
An example of a coarse property that is not regular is coarse embeddability in reals.
Notice that a wedge of two copies of reals does not embed into reals. That also means that satisfying a coarse property in the sense of Guentner is more general than discussed in this paper.
\end{Observation}

\begin{Observation}
Suppose $\mathscr{P}$ is a regular coarse property and $\{X_s\}_{s\in S}$ is a family of $\infty$-pseudo-metric spaces. If $S$ is finite and each $X_s$ has $\mathscr{P}$, then
 $\{X_s\}_{s\in S}$ 
has $\mathscr{P}$ uniformly.
\end{Observation}

\begin{Proposition} \label{BallEnlargementProp}
Suppose $\mathscr{P}$ is a regular coarse property.
Given a family $\{X_s\}_{s\in S}$ of $\infty$-pseudo-metric spaces, given $r > 0$, and given subsets
$A_s$ of $X_s$ for each $s\in S$, if $\{A_s\}_{s\in S}$ satisfies $\mathscr{P}$ uniformly, then so does
the family $\{B(A_s,r)\}_{s\in S}$.
\end{Proposition}
\begin{proof}
Given $x_s\in B(A_s,r)$ for each $s\in S$, choose $y_s\in A_s$ so that $d(x_s,y_s)\leq r$ for each $s\in S$. Notice that any function $f:\bigvee\limits_{s\in S} (A_s,y_s)\to \bigvee\limits_{s\in S} (B(A_s,r),x_s)$,
which is the identity on $A_s\setminus \{y_s\}$ and sends $y_s$ to $x_s$, is a coarse equivalence.
\end{proof}

\begin{Proposition}\label{BoundedUnionProp}
Suppose $\mathscr{P}$ is a regular coarse property.
Given a family $\{X_s\}_{s\in S}$ of subsets of an $\infty$-pseudo-metric space $X$
that satisfies $\mathscr{P}$ uniformly,  the union $\bigcup\limits_{s\in S} X_s$ satisfies $\mathscr{P}$
if there is a point $x_0\in \bigcap\limits_{s\in S} X_s $ with the property that for each $r > 0$
there is $t > 0$ so that the family $\{X_s\setminus B(x_0,t)\}_{s\in S}$ is $r$-disjoint.
\end{Proposition}
\begin{proof}
Consider $f:\bigvee\limits_{s\in S} (X_s,x_0)\to \bigcup\limits_{s\in S} X_s$ that is the identity
on each $X_s$. Notice $f$ is a large scale equivalence.
\end{proof}

\section{Coarse properties via partitions of unity}
There are three basic ways of defining coarse properties:\\
1. Via covers,\\
2. Via $R$-disjoint families,\\
3. Via partitions of unity.

This section is devoted to the third thread of defining coarse properties and to comparing it to two other threads. Also, we are interested if coarse properties defined that way are regular.

\begin{Definition}
Suppose $X$ is an $\infty$-pseudo-metric space. We say its coarse structure is \textbf{determined by partitions of unity} (by \textbf{point-finite partitions of unity}, respectively) if for every $R > 0$
there is a partition of unity (a point-finite partition of unity, respectively) whose strata are uniformly bounded and of Lebesgue number at least $R$.
\end{Definition}

\begin{Observation}\label{RDisjointObservation}
Suppose $X$ is an $\infty$-pseudo-metric space. If, for each $R > 0$, $X$ can be expressed as a finite union $X=\bigcup\limits_{i=1}^m X_i$ and each $X_i$ is the union of an $R$-disjoint family of uniformly bounded subsets of $X$,
then the coarse structure of $X$ is determined by point-finite partitions of unity.
\end{Observation}
\begin{proof}
Express $X$ as a finite union $X=\bigcup\limits_{i=1}^m X_i$, where each $X_i$ is the union of an $3R$-disjoint family $\{X_i^j\}_{j\in S(i)}$ of uniformly bounded subsets of $X$. The partition of unity obtained by normalizing the
sum of characteristic functions of sets $B(X_i^j,R)$, $i\leq m$, $j\in S(i)$, has strata that form a uniformly bounded family of Lebesgue number at least $R$.
\end{proof}

\begin{Definition}
Suppose $X$ is an $\infty$-pseudo-metric space. We say $X$ is \textbf{exact} if for every $R, \epsilon > 0$
there is an $(\epsilon,R)$-continuous partition of unity whose strata are uniformly bounded.
\end{Definition}

\begin{Definition}
Suppose $X$ is an $\infty$-pseudo-metric space. We say $X$ is \textbf{large scale paracompact} if for every $R, \epsilon > 0$
there is an $(\epsilon,R)$-continuous point-finite partition of unity whose strata are uniformly bounded and of Lebesgue number at least $R$.
\end{Definition}

\begin{Definition}
Suppose $X$ is an $\infty$-pseudo-metric space.
Given a regular coarse property $\mathscr{P}$,
the \textbf{asymptotic dimension} of $X$ with respect to $\mathscr{P}$ is at most $n$ if 
if for every $R, \epsilon > 0$
there is an $(\epsilon,R)$-continuous partition of unity $\phi$ whose strata have $\mathscr{P}$ uniformly and the dimension of $\phi$ is at most $n$.
\end{Definition}

\begin{Theorem}\label{CharAsdimWRTP}
Given a regular coarse property $\mathscr{P}$, given $n\ge 0$, and given an $\infty$-pseudo-metric space $X$, the following conditions are equivalent:\\
1. The asymptotic dimension of $X$ with respect to $\mathscr{P}$ is at most $n$.\\
2. For all real numbers $R$,
$X$ decomposes as a finite union $\bigcup\limits_{k=0}^n X_k$
and each $X_i$ is an $R$-disjoint union of sets having $\mathscr{P}$ uniformly.\\
3. For all real numbers $R$,
$X$ has a cover of Lebesgue number at least $R$, multiplicity at most $n+1$, and having $\mathscr{P}$ uniformly.
\end{Theorem}
\begin{proof}
1)$\implies$2). Notice there is $\epsilon > 0$ such that for any simplicial complex $K$ of dimension at most $n$ the second barycentric subdivision $K''$ of $K$ has the property that two closed stars of different barycenters $\sigma_1$, $\sigma_2$ of $K$ are $\epsilon$-disjoint if $\dim(\sigma_1)=\dim(\sigma_2)$. Therefore, given an $(R,\epsilon)$-continuous partition of unity $\phi$ on $X$
of dimension at most $n$, one defines $X_i$ as the set of points $x\in X$ whose image $\phi(x)$ lands
in a closed star of the barycenter of some $i$-simplex of $K$, where $K$ is the $n$-skeleton
of the full complex spanned by the set of indices of $\phi$.

2)$\implies$3). Decompose $X$ as a finite union $\bigcup\limits_{k=0}^n X_k$
and each $X_i$ is an $3R$-disjoint union of sets having $\mathscr{P}$ uniformly.
Consider the family of $R$-balls around all the sets that are listed in decompositions of $X_i$, $i\leq n$. Apply \ref{BallEnlargementProp}.

3)$\implies$1). Apply \ref{ExistenceOfREpsilonPUs}.
\end{proof}

\begin{Definition}\label{DisjointDecomposition}
Let $\mathscr{P}$ be a regular coarse property, $\{R_i\}_{i\ge 1}$ is a sequence of real numbers, and $m > 0$ is an integer. An $\infty$-pseudo-metric space $X$ has 
an \textbf{$(m,\{R_i\}_{i\ge 1},\mathscr{P})$-decomposition} if $X$ decomposes as the union of $m$-subsets $X_i$, $1\leq i\leq m$,
and each $X_i$ is a union of an $R_i$-disjoint family satisfying $\mathscr{P}$ uniformly.
\end{Definition}

\begin{Lemma}\label{BasicLemmaAboutRepsilonPDecomposition}
Let $\mathscr{P}$ be a regular coarse property, $\{R_i\}_{i\ge 1}$ is a sequence of real numbers, and $m > 0$ is an integer, and $\{X_s\}_{s\in S}$ is a family of $\infty$-pseudo-metric spaces.
$\bigvee\limits_{s\in S}(X_s,x_s)$ has an $(m,\{R_i\}_{i\ge 1},\mathscr{P})$-decomposition in the following two cases:\\
1. Each $X_s$ has an $(m,\{R_i\}_{i\ge 1},\mathscr{P})$-decomposition and $S$ is finite,\\
2. All $X_s$, $s\in S$, are subsets of an $\infty$-pseudo-metric space $X$ that has
an $(m,\{R_i\}_{i\ge 1},\mathscr{P})$-decomposition.
\end{Lemma}
\begin{proof}
Assume each $X_s$ decomposes as the union of $m$-subsets $X_s^i$, $i\leq m$,
and each $X_s^i$ is a union of an $R_i$-disjoint family satisfying $\mathscr{P}$ uniformly. In Case 2, those decompositions are induced from a decomposition of $X$.

Let $R=\max(R_i | i\leq m)$. Let $B$ be the ball around the basepoint of $X$ of radius $R$.
Given $i\leq m$, add all elements of decompositions of $X_s^i$ intersecting $B$.
By \ref{BoundedUnionProp}, that union satisfies $\mathscr{P}$.
$X^i$ consists of that set plus all the elements of decompositions of $X_s^i$ that do not intersect $B$.
\end{proof}

\begin{Corollary}
Given a regular coarse property $\mathscr{P}$, and given $n\ge 0$, the coarse property
of having asymptotic dimension with respect to $\mathscr{P}$ at most $n$ is a regular coarse property.
\end{Corollary}
\begin{proof}
Having asymptotic dimension with respect to $\mathscr{P}$ at most $n$ is the same
as having an $(n+1,\{R\}_{i\ge 1},\mathscr{P})$-decomposition for any constant sequence
$\{R\}_{i\ge 1}$ (see \ref{CharAsdimWRTP}).
Apply \ref{BasicLemmaAboutRepsilonPDecomposition}.
\end{proof}

\section{Decomposition trees}
In this section we introduce a concept that formalizes the way finite decomposition complexity was introduced in \cite{GTY1}  (see also \cite{GTY2}) and also applies to the way straight finite decomposition complexity was introduced in \cite{DZ}.

\begin{Definition}\label{decomposition tree}
Suppose $X$ is an $\infty$-pseudo-metric space and $\{(R_i,n_i)\}_{i\ge 1}$ is an infinite
sequence of pairs on natural numbers.\\
An \textbf{$\{(R_i,n_i)\}_{i\ge 1}$-decomposition tree} $\mathcal{T}$ on a set $X$ is a rooted tree whose vertices are subsets of $X$ satisfying the following conditions:\\
1. The subset at the root of $\mathcal{T}$ is the whole $X$,\\
2. If there is an edge from $A$ at depth $n$ to $B$ at depth $n+1$, then $B\subset A$, \\
3. Each non-leaf $A$ at even depth $i$ has at most $n_{i+1}$ children and is the union of its children,\\
4. For each non-leaf $A$ at odd depth $i$, the family of its children is $R_i$-disjoint and covers $A$.
\end{Definition}

\begin{Theorem}\label{ConvertDecompositionTreeToPUTree}
Suppose $\mathscr{P}$ is a regular coarse property and  $X$ is an $\infty$-pseudo-metric space.
If there is a sequence $\{(n_i)\}_{i\ge 1}$ of natural numbers such that
for each infinite sequence $\{(R_i)\}_{i\ge 1}$ of real numbers
$X$ has an $\{(R_i,n_i)\}_{i\ge 1}$-decomposition tree of finite height whose leaves satisfy $\mathscr{P}$
uniformly, then for each infinite sequence of pairs of positive reals $\{(\epsilon_i,R_i)\}_{i\ge 1}$,
there is a tree of partitions of unity on $X$ whose partition of unity has strata satisfying $\mathscr{P}$ uniformly and  each vertex of $\mathcal{T}$ at depth $i$ is $(\epsilon_i,R_i)$-continuous.
\end{Theorem}
\begin{proof}
Given an infinite sequence of pairs of positive reals $\{(\epsilon_i,R_i)\}_{i\ge 1}$,
define $S_k=\sum\limits_{i=1}^k \frac{8n_i\cdot R_i}{\epsilon_i}$ and
pick an $\{(S_i,n_i)\}_{i\ge 1}$-decomposition tree of finite height whose leaves satisfy $\mathscr{P}$ uniformly.
Replace it by an $\{(\frac{4n_i\cdot R_i}{\epsilon_i},n_i)\}_{i\ge 1}$-decomposition tree $\mathscr{T}$ of the same height whose leaves satisfy $\mathscr{P}$ uniformly by induction as follows:
In Step 1, replace each vertex $A$ by $B(A,\frac{4n_1\cdot R_1}{\epsilon_1})$.
In Step 2, replace each new vertex $C$ at depth $2$ or higher by $B(C,\frac{4n_2\cdot R_2}{\epsilon_2})$
with the proviso that the balls are with respect to vertices at height $1$. Continue until one reaches leaves of the tree.

Construct a tree of partitions of unity using \ref{ExistenceOfREpsilonPUs} and notice it has required properties.
\end{proof}

\section{Union permanence}

Union permanence for coarse invariants as defined by Guentner \cite{GuentnerPermanence} needs to be modified since we consider more general spaces than those in \cite{GuentnerPermanence}.

\begin{Definition}
Suppose $\mathscr{P}$ is a regular coarse property. $\mathscr{P}$ has the \textbf{$\infty$-union permanence} if $\bigcup\limits_{s\in S} X_s$ has $\mathscr{P}$ whenever the family
$\{X_s\}_{s\in S} $ has $\mathscr{P}$ uniformly and the distance between points of $X_t$ and $X_s$
is infinity for all $s\ne t\in S$.
\end{Definition}

\begin{Observation}
The property of being coarsely equivalent to one-point space does not have $\infty$-union permanence. The same is true for the property of being coarsely embeddable in a Hilbert space.
\end{Observation}

\begin{Definition}
Suppose $\mathscr{P}$ is a regular coarse property. $\mathscr{P}$ has the \textbf{finite union permanence} if $X\cup Y$ has $\mathscr{P}$ whenever $X, Y$ have $\mathscr{P}$
and $X\cap Y\ne\emptyset$.
\end{Definition}

\begin{Definition}\label{UnionPermanence}
Suppose $\mathscr{P}$ is a regular coarse property. $\mathscr{P}$ has the \textbf{union permanence} if $X$ has $\mathscr{P}$ whenever for each
$R, \epsilon > 0$ there is an $(R,\epsilon)$-continuous partition of unity $\phi$ on $X$ whose strata have $\mathscr{P}$
uniformly and the range of $\phi$ is contained in a hedgehog, i.e. a rooted tree of height $1$.
\end{Definition}

The next result shows that our definition of union permanence \ref{UnionPermanence} corresponds to the one given by Guentner \cite{GuentnerPermanence} which amounts to 2) below.

\begin{Proposition}
Suppose $\mathscr{P}$ is a regular coarse property. $\mathscr{P}$ has the union permanence
if and only if the following conditions are equivalent for every $\infty$-pseudo-metric space $X$:\\
1. $X$ has property $\mathscr{P}$,\\
2. For each
$R > 0$ there is a subspace $Y$ of $X$ having $\mathscr{P}$ such that
$X\setminus Y$ can be represented as an $R$-disjoint union of sets having $\mathscr{P}$ uniformly.
\end{Proposition}
\begin{proof}
Suppose 2) is satisfied and $R, \epsilon > 0$. Choose a subspace $Y$ of $X$ having $\mathscr{P}$ such that
$X\setminus Y$ can be represented as an $\frac{24R}{\epsilon}$-disjoint union of sets $\{Y_s\}_{s\in S}$ having $\mathscr{P}$ uniformly. Consider the cover of $X$ equal to $\{B(Y,\frac{12R}{\epsilon})\}\cup \{B(Y_s,\frac{12R}{\epsilon})\}_{s\in S}$. The nerve of that cover is contained in a hedgehog
and the partition of unity constructed in \ref{ExistenceOfREpsilonPUs} is $(R,\epsilon)$-continuous.

Conversely, suppose that for each
$R, \epsilon > 0$ there is an $(R,\epsilon)$-continuous partition of unity $\phi$ on $X$ whose strata have $\mathscr{P}$
uniformly and the range of $\phi$ is a hedgehog. It suffices to consider $\epsilon=0.5$.
Let $Y$ be the stratum of the root of the range of $\phi$. $X\setminus Y$ is the union
of point-inverses of the vertices of the range of $\phi$ and that family is $R$-disjoint.
\end{proof}

\begin{Corollary}
Suppose $\mathscr{P}$ is a regular coarse property.
If $\mathscr{P}$ has the union permanence, then it has the $\infty$-permanence and the finite union permanence.
\end{Corollary}
\begin{proof}
Suppose the family
$\{X_s\}_{s\in S} $ has $\mathscr{P}$ uniformly and the distance between points of $X_t$ and $X_s$
is infinity for all $s\ne t\in S$. In that case the partition of unity consisting of characteristic functions of sets $X_s$ is $(\epsilon,R)$-continuous for all $R, \epsilon > 0$ and its range is a $0$-dimensional simplicial complex. Hence, $\bigcup\limits_{s\in S} X_s$ has $\mathscr{P}$.

Suppose $X, Y$ have $\mathscr{P}$ and $X\cap Y\ne\emptyset$. The family $\{X\setminus Y\}$ is trivially $R$-disjoint.
Hence, if $X\cup Y$ has $\mathscr{P}$.
\end{proof}

\section{Probabilistic decomposition complexity}
In this section we introduce a coarse property that generalizes exactness and allows for easy 
proofs that finite decomposition complexity (or straight finite decomposition complexity) implies Property A of G.Yu.

\begin{Definition}\label{Probabilistic decomposition complexity}
Suppose $\mathscr{P}$ is a regular coarse property.
An $\infty$-pseudo-metric space has the \textbf{probabilistic decomposition complexity} with respect to $\mathscr{P}$ if for each
$R, \epsilon > 0$ there is an $(R,\epsilon)$-continuous partition of unity on $X$ whose strata have $\mathscr{P}$
uniformly.
\end{Definition}

\begin{Theorem}\label{LongTreesOfPUsToShortTrees}
Suppose $\mathscr{P}$ is a regular coarse property and $X$ is an $\infty$-pseudo-metric space.
$X$ has probabilistic decomposition complexity with respect to $\mathscr{P}$ if for each
sequence $R, \epsilon_n > 0$ there is a tree $\mathscr{T}$ of partitions of unity on $X$ with the following properties:\\
1. the strata of leaves of $\mathscr{T}$ have $\mathscr{P}$
uniformly.\\
2. The vertices of $\mathscr{T}$ at depth $n$ are $(R,\epsilon_n)$-continuous.\\
3. $\mathscr{T}|\{x\}$ induces a probability tree for each $x\in X$.
\end{Theorem}
\begin{proof}
Given $R, \epsilon > 0$ choose a sequence $\{\epsilon_i\}_{i\ge 0}$ of positive real numbers such that
$\epsilon > \sum\limits_{i=0}^\infty \epsilon_i$. The tree $\mathscr{T}$ of partitions of unity on $X$ with the  properties as above induces, in view of \ref{FundThmOnRootsOfPUs}, an $(R,\epsilon)$-continuous partition of unity on $X$ whose strata have $\mathscr{P}$
uniformly.
\end{proof}

The following definition generalizes straight finite decomposition complexity of
Dranishnikov and Zarichnyi \cite{DZ} who consider only the case of $n_i=2$
and $\mathscr{P}$ is the property of being bounded. Also, it generalizes the concept
of countable asymptotic dimension from \cite{Dyd1}. 

\begin{Definition}\label{straight finite decomposition complexity}
Suppose $\mathscr{P}$ is a regular coarse property.
An $\infty$-pseudo-metric space has a \textbf{countable asymptotic dimension} with respect to $\mathscr{P}$ if there is a sequence of integers $n_i\ge 0$ such that for each sequence
$R_i > 0$, $X$ has an $\{(R_i,n_i)\}_{i\ge 1}$-decomposition tree whose leaves satisfy $\mathscr{P}$
uniformly.
\end{Definition}

\begin{Theorem}
Suppose $\mathscr{P}$ is a regular coarse property. If an $\infty$-pseudo-metric space $X$
has countable asymptotic dimension with respect to $\mathscr{P}$, then it has
probabilistic decomposition complexity with respect to $\mathscr{P}$.
\end{Theorem}
\begin{proof}
Apply \ref{ConvertDecompositionTreeToPUTree} and \ref{LongTreesOfPUsToShortTrees}.
\end{proof}

\section{Exactness and large scale paracompactness}
Notice that large scale paracompactness is a special case of exactness. In this section we give a sufficient condition for the two properties being equivalent.

\begin{Proposition}\label{ConnectionBetweenExactnessAndLSParacompactness}
Suppose $\mathscr{P}$ is a regular coarse property and $X$ is an $\infty$-pseudo-metric space.
If $X$ has probabilistic decomposition complexity with respect to $\mathscr{P}$ and its coarse structure is determined by (point-finite) partitions of unity, then for each
$R, \epsilon > 0$ there is a (point-finite) $(R,\epsilon)$-continuous partition of unity on $X$ whose strata have $\mathscr{P}$
uniformly and form a cover of $X$ of Lebesgue number at least $R$.
\end{Proposition}
\begin{proof}
Suppose $R, \epsilon > 0$. Pick an $(R,\epsilon/4)$-continuous partition of unity $\phi=\{\phi_s\}_{s\in S}$ on $X$ whose strata have $\mathscr{P}$ uniformly. Also, pick a partition of unity 
$\psi=\{\psi_t\}_{t\in T}$ on $X$ whose strata are uniformly bounded and of Lebesgue number at least $R$. Assume $S$ and $T$ are disjoint and define
$f(x)$ as $(1-\epsilon/4)\cdot \phi(x)+0.25\cdot\epsilon\cdot \psi(x)$.
Notice $f$ is $(R,\epsilon)$-continuous and its strata have $\mathscr{P}$
uniformly by \ref{BallEnlargementProp}.

In the point-finite case we can trim $\phi$ to a point-finite partition of unity by
choosing, for each $x\in X$, a finite subset $S(x)$ of $S$ such that
$\sum\limits_{s\in S(x)}\phi_s(x) > 1-\epsilon/8$ and allocating
$\sum\limits_{s\notin S(x)}\phi_s(x)$ to a particular $s\in S(x)$.
\end{proof}

Recall that a map $f$ from $X$ to a metric space $Y$ is $(\delta,\delta)$-Lipschitz if $d_Y(f(x),f(y))\leq \delta\cdot d_X(x,y)+\delta$ for all $x,y\in X$.

The following is shown in \cite{Dyd1} (see Theorem 6.4) in case of metric spaces $X$. The same proof works for $\infty$-pseudo-metric spaces.

\begin{Theorem}\label{ExactnessViaPUExtension}
Suppose $X$ is an $\infty$-pseudo-metric space.
Choose a set $S$ whose cardinality is larger that $card(X\times \mathbb{N})$. 
$X$ being exact is equivalent to the existence of functions
 $\alpha: (0,\infty)\to (0,\infty)$, $M: (0,\infty)\times (0,\infty)\to (0,\infty)$ such that for any $K > 0$, any $(\alpha(\delta),\alpha(\delta))$-Lipschitz map $f:A\subset X\to \Delta(S)$ ($\delta > 0$) that is $K$-cobounded (i.e. point-inverses of stars of vertices of $\Delta(S)$ have diameter at most $K$),
extends to a $(\delta,\delta)$-Lipschitz map $g:X\to \Delta(S)$ that is $M(\delta,K)$-cobounded.
\end{Theorem}

\begin{Proposition}\label{ExactnessIsRegular}
Being exact is a regular coarse property.
\end{Proposition}
\begin{proof}
Let $\mathscr{P}$ be the class of exact $\infty$-pseudo-metric spaces.
According to \ref{RegularCoarseProperty}
we need to check the following conditions:\\
1. If an $\infty$-pseudo-metric space $X$ satisfies $\mathscr{P}$ and $A_s$ is an isometric copy of $X$ for $s\in S$, then $\bigvee\limits_{s\in S} (A_s,x_s)$ has $\mathscr{P}$ for any choice of $x_s\in A_s$.\\
2. If each element of a finite family $\{X_s\}_{s\in S}$ of $\infty$-pseudo-metric spaces satisfies $\mathscr{P}$, then $\bigvee\limits_{s\in S} (A_s,x_s)$ has $\mathscr{P}$ for some choice of $x_s\in A_s$.

We will only show 1) as the proof of 2) is similar. Choose a set $T$ whose cardinality is larger that $card(X\times \mathbb{N})$  and consider functions  $\alpha: (0,\infty)\to (0,\infty)$, $M: (0,\infty)\times (0,\infty)\to (0,\infty)$ as in \ref{ExactnessViaPUExtension}.
Given $\epsilon > 0$ we choose
for each $s\in S$ an extension $\phi_s:A_s\to \Delta(T)$ of the map sending $x_s$
to a fixed vertex $t_0$ of $\Delta(T)$, an extension which is $(\epsilon/2,\epsilon/2)$-Lipschitz and is $K$-cobounded with $K$ the same for all $s\in S$.
Pasting all $\phi_t$ gives an $(\epsilon,\epsilon)$-Lipschitz function from
$\bigvee\limits_{s\in S} (A_s,x_s)$ to $\Delta(\bigvee\limits_{s\in S}(T,t_0))$ that is $2K$-cobounded.
\end{proof}

\begin{Proposition}\label{WeakParacompactnessIsRegular}
Having the coarse structure determined by point-finite partitions of unity is a regular coarse property.
\end{Proposition}
\begin{proof}
Similar to that of \ref{ExactnessIsRegular}.
\end{proof}

\begin{Corollary}
Being large scale paracompact is a regular coarse property.
\end{Corollary}
\begin{proof}
By \ref{ConnectionBetweenExactnessAndLSParacompactness} being large scale paracompact means being exact and having the large scale structure determined by point-finite partitions of unity. Both classes are regular by \ref{ExactnessIsRegular} and \ref{WeakParacompactnessIsRegular},
so being large scale paracompact is a regular coarse property.
\end{proof}

\section{Coarse embeddings in Hilbert spaces}

In this section we prove that coarse embeddability in Hilbert spaces is a regular coarse property. Obviously, in this case we restrict ourselves to pseudo-metric spaces.

The following is a generalization of a result of M.Holloway \cite{Hol}.
\begin{Theorem}
Suppose $X$ is a pseudo-metric space. The following conditions are equivalent:\\
1. $X$ coarsely embeds in $l_1(A)$ for some set $A$,\\
2. There is a constant $c > 0$ such that for all $R, \epsilon > 0$ there is $S > 0$ with the property that for any $T > S$ there is a set $C$ and a function $f$ from $X$ to the unit sphere of $l_1(C)$ so that
$|f(x)-f(y)|\leq \epsilon$ if $d(x,y)\leq R$ and $|f(x)-f(y)|\geq c$ if $T \geq d(x,y)\geq S$. 
\end{Theorem}
\begin{proof}
2)$\implies$1). For each $i\ge 0$ put $R_i=i$ and $\epsilon_i=1/2^i$. Find a sequence of natural numbers $S(i)$ for those values and adjust it, if necessary, to be strictly increasing and $S(i+1)-S(i)$ is divergent to infinity. 
Find functions $f_i$ from $X$ to the unit sphere of $l_1(C_i)$ such that
$|f(x)-f(y)|\leq \epsilon_i$ if $d(x,y)\leq R_i$ and $|f(x)-f(y)|\geq c$ if $S((i+1)^2) \geq d(x,y)\geq S(i)$. 
We may assume sets $C_i$ are pairwise disjoint. Pick $x_0\in X$ and define
$f:X\to l_1(\bigcup\limits_{i=0}^\infty C_i)$ by the formula
$$f(x)=\sum\limits_{i=0}^\infty f_i(x)-f_i(x_0).$$
Notice $d(x,y)\leq n$, $n$ a natural number, implies $|f(x)-f(y)|\leq 2n+1$. Therefore $f$ is large scale continuous (or bornologous).
If $S((i+1)^2) \geq d(x,y)\geq S(i^2)$, then $|f_k(x)-f_k(y)|\ge c$ for all $i^2\ge k \ge i$, resulting in
$|f(x)-f(y)|\ge (i^2-i)\cdot c$. Therefore $d(x_n,y_n)\to \infty$ implies $|f(x_n)-f(y_n)|\to \infty$
and $f$ is a coarse embedding.
\end{proof}

\begin{Corollary}
The property of being coarsely embeddable in a Hilbert space is a regular coarse property.
\end{Corollary}
\begin{proof}
Coarse embeddability in $l_1(A)$ for some set $A$ is equivalent to coarse embeddability in a Hilbert space (see \cite{NowakYu}). 

Suppose a pseudo-metric space $X$ coarsely embeds in $l_1(A)$ for some set $A$ and $A_s$ is an isometric copy of $X$ for $s\in S$. Pick a constant $c > 0$ such that for all $R, \epsilon > 0$ there is $S > 0$ with the property that for any $T > S$ there is a set $C$ and a function $f$ from $X$ to the unit sphere of $l_1(C)$ so that
$|f(x)-f(y)|\leq \epsilon/2$ if $d(x,y)\leq R$ and $|f(x)-f(y)|\geq 2c$ if $T \geq d(x,y)\geq S$. 

For each $s\in S$ we may find $f_s:A_s\to l_1(C_s)$, $C_s$ being a copy of $C$, with the same properties as $f$
and sending a fixed $x_s\in A_s$ to the Dirac function $\delta_{c_s}$ for some $c_s\in C_s$. Those functions induce $g:\bigvee\limits_{s\in S} (A_s,x_s)\to l_1(\bigvee\limits_{s\in S} (C_s,c_s))$ with the property that there is $S > 0$ so that for any $T > 2S$ there is a set $C$ and a function $f$ from $X$ to the unit sphere of $l_1(C)$ so that
$|g(x)-g(y)|\leq \epsilon$ if $d(x,y)\leq R$ and $|g(x)-g(y)|\geq c$ if $T \geq d(x,y)\geq 2S$. 

The same way we can check the remaining condition in the Definition 
\ref{RegularCoarseProperty} of regular coarse properties.
\end{proof}

\section{Asymptotic Property C}
In this section we generalize Dranishnikov's asymptotic property C (see \cite{Dra1}  and \cite{DZ})
and prove that it leads to regular coarse properties. Also, we generalize the result that spaces with asymptotic property C have Property A of G.Yu.

\begin{Definition}\label{AsymptoticPropertyCDef}
Suppose $\mathscr{P}$ is a regular coarse property and $X$ is an $\infty$-pseudo-metric space.
$X$ has \textbf{asymptotic property C} with respect to $\mathscr{P}$ if for each 
sequence $\{R_i\}_{i\ge 1}$ there is a natural number $n$ and a decomposition
$X=\bigcup\limits_{i=1}^m X_i$ such that each $X_i$ is the union of an $R_i$-disjoint family 
of subsets of $X$ having $\mathscr{P}$ uniformly.
\end{Definition}

\begin{Corollary}
Suppose $\mathscr{P}$ is a regular coarse property.
Asymptotic Property C with respect to $\mathscr{P}$ is a regular coarse property.
\end{Corollary}
\begin{proof}
It follows directly from \ref{BasicLemmaAboutRepsilonPDecomposition}
\end{proof}

\begin{Theorem}
Suppose $\mathscr{P}$ is a regular coarse property and $X$ is an $\infty$-pseudo-metric space.
If $X$ has asymptotic property C with respect to $\mathscr{P}$, then $X$ has probabilistic decomposition 
complexity with respect to $\mathscr{P}$.
\end{Theorem}
\begin{proof}
Apply \ref{LongTreesOfPUsToShortTrees} and \ref{ConvertDecompositionTreeToPUTree}.
\end{proof}

In view of \ref{RDisjointObservation} the definition \ref{AsymptoticPropertyCDef}
can be generalized as follows:

\begin{Definition}\label{BasicDecompositionComplexityDef}
Suppose $\mathscr{P}$ is a regular coarse property and $X$ is an $\infty$-pseudo-metric space.
$X$ has \textbf{basic decomposition complexity} with respect to $\mathscr{P}$ if for each 
 $R > 0$ there is a natural number $n$ and a decomposition
$X=\bigcup\limits_{i=1}^m X_i$ such that each $X_i$ is the union of an $R$-disjoint family 
of subsets of $X$ having $\mathscr{P}$ uniformly.
\end{Definition}

In the case of  $\mathscr{P}$ meaning being bounded \ref{BasicDecompositionComplexityDef}
is a generalization of bounded geometry.

\begin{Proposition}
Suppose $X$ is an $\infty$-pseudo-metric space.
If $X$ has bounded geometry, then  for each 
 $R > 0$ there is a natural number $n$ and a decomposition
$X=\bigcup\limits_{i=1}^m X_i$ such that each $X_i$ is the union of an $R$-disjoint uniformly bounded family 
of subsets of $X$.
\end{Proposition}
\begin{proof}
Let $X_1$ be a subset of $X$ that is maximal with respect to the following property:
$d(x,y) > R$ whenever $x\ne y\in X_1$. By induction, define $X_i$, $i > 2$, as a subset
subset of $X\setminus X_{i-1}$ that is maximal with respect to the following property:
$d(x,y) > R$ whenever $x\ne y\in X_i$.
Let $n$ be a number such that any ball $B(x,2R)$ has less than $n$ points.
Notice $X_{n+1}=\emptyset$. Indeed, if $x\in X_{n+1}$, then
each set $X_i$, $i \leq n$, contains at least one point from $B(x,2R)$, a contradiction.
\end{proof}

\section{Stability of coarse properties}

Suppose $\mathscr{P}$ is a regular coarse property.
There are two kinds of stability for $\mathscr{P}$:\\
1. The class of $\infty$-pseudo-metric spaces
 that have probabilistic decomposition complexity with respect to $\mathscr{P}$ coincides with the class of  $\infty$-pseudo-metric spaces that have property $\mathscr{P}$.\\
2. Every $\infty$-pseudo-metric space $X$ must have $\mathscr{P}$ once all its bounded subspaces
have $\mathscr{P}$ uniformly.

Notice that $\mathscr{P}$  being stable in the sense of 1) implies union permanence for $\mathscr{P}$ (see \ref{UnionPermanence}). 

\begin{Proposition}
Let $X$ be an $\infty$-pseudo-metric space. If $X$ has probabilistic decomposition complexity with respect to exact spaces, then $X$ is exact.
\end{Proposition}
\begin{proof}
Apply \ref{FundLemmaOnRootedTreesOfPUs} as follows:
Given $R, \epsilon > 0$, choose an $(\epsilon,R)$-continuous partition of unity $\phi$ whose strata are exact uniformly.
Then, on each stratum, choose an $(\epsilon,R)$-continuous partition of unity whose strata are uniformly bounded by some $M > 0$
that is the same for all strata of $\phi$.
\end{proof}

\begin{Theorem}
Let $X$ be an $\infty$-pseudo-metric space. If all bounded subsets of $X$ are exact uniformly, then $X$ is exact.
\end{Theorem}
\begin{proof}
By applying \ref{ExactnessViaPUExtension} there exist functions $\alpha: (0,\infty)\to (0,\infty)$, $M: (0,\infty)\times (0,\infty)\to (0,\infty)$ that work for all bounded subsets of $X$ simultanously.

Given $\delta > 0$ pick $x_0\in X$, consider annuli $A_n$, $n\ge 1$, defined as $\{x\in X | (n-1)\cdot R\leq d(x,x_0)\leq n\cdot R\}$ for $R$ sufficiently large. For each $n\ge 0$, choose
$f_n:A_n\to \Delta(S_n)$
that is $(\alpha(\delta),\alpha(\delta))$-Lipschitz and $K$-cobounded for some $K$ uniform to all $n$, and $S_n\subset S$ are mutually disjoint. Pasting all $n$ with $n$ even
gives a function $f$ to $\Delta(S)$ that is $(\alpha(\delta),\alpha(\delta))$-Lipschitz and $K$-cobounded.
Then, for each $n\ge 1$, we extend $f$ restricted to $A_{2n-2}\cup A_{2n}$ over $A_{2n-2}\cup A_{2n-1}\cup A_{2n}$
so that the extension $g_n$ is $(\delta,\delta)$-Lipschitz and $M(\delta,K)$-cobounded. Finally, pasting all $g_n$'s gives
$g:X\to \Delta(S)$ that is $(\delta,\delta)$-Lipschitz $2M(\delta,K)$-cobounded.
\end{proof}

\begin{Corollary}
Suppose $\mathscr{P}$ is a regular coarse property.
If $\mathscr{P}$ implies exactness, then being stable in sense 1) is equivalent to being stable in sense 2).
\end{Corollary}

\begin{Proposition}
Suppose $\mathscr{P}$ is a regular coarse property and $X$ is an $\infty$-pseudo-metric space.
If all bounded subsets of $X$ have countable asymptotic dimension with respect to $\mathscr{P}$
uniformly, then $X$ has countable asymptotic dimension with respect to $\mathscr{P}$.
\end{Proposition}
\begin{proof}
Let's consider the case of $X$ being pseudo-metric with a base point $x_0$. The proof in the general case of  $\infty$-pseudo-metric spaces is similar.
Let $n_i\ge 0$ be a sequence of integers such that for each sequence
$R_i > 0$, all bounded subsets of $X$ have an $\{(R_i,n_i)\}_{i\ge 1}$-decomposition tree whose leaves satisfy $\mathscr{P}$
uniformly (see \ref{straight finite decomposition complexity}) and the height of all such trees is bounded by some $M > 0$.
Define $\{m_i\}_{i\ge 1}$ by $m_1=2$ and $m_{i+1}=n_i$ for $i\ge 1$
Given a sequence $\{R_i\}_{i\ge 1}$ of positive real numbers define 
the annulus $A_n$, $n\ge 1$, as $\{x\in X | (n-1)\cdot R_1\leq d(x,x_0)\leq n\cdot R_1\}$.
Build an $\{(R_i,m_i)\}_{i\ge 1}$-decomposition tree whose leaves satisfy $\mathscr{P}$
uniformly as follows:\\
1. Put $X$ at the root of it,\\
2. $X$ has two children: the union of all odd-numbered annuli $A_n$ and the union
of all even-numbered annuli $A_n$,\\
3. Follow with an $\{(R_i,n_i)\}_{i\ge 2}$-decomposition tree of each annulus.
\end{proof}

\begin{Question}
Let $\mathscr{P}$  be the property of having Asymptotic Property C.
If all bounded subsets of $X$ have $\mathscr{P}$ uniformly, then does $X$
have Asymptotic Property C?
\end{Question}

\begin{Question}
Let $\mathscr{P}$  be the property of having Asymptotic Property C.
If $X$ has probabilistic decomposition complexity with respect to $\mathscr{P}$, then does $X$
have Asymptotic Property C?
\end{Question}

\begin{Question}
Let $\mathscr{P}$  be the property of being coarsely embeddable in a Hilbert space.
If $X$ has probabilistic decomposition complexity with respect to $\mathscr{P}$, then is $X$
coarsely embeddable in a Hilbert space?
\end{Question}

\begin{Question}
Let $\mathscr{P}$  be the property of being coarsely embeddable in a Hilbert space.
If the family of bounded subsets of $X$ has $\mathscr{P}$ uniformly, then is $X$
coarsely embeddable in a Hilbert space?
\end{Question}


\begin{thebibliography}{99}

\bibitem{BC} Paul Baum and Alain Connes. K-theory for discrete groups. In Operator algebras and applications, Vol. 1, volume 135 of London Math. Soc. Lecture Note Ser., pages 1--20. Cambridge Univ. Press, Cambridge, 1988.

\bibitem{BG} Susan Beckhardt, Boris Goldfarb, Extension properties of asymptotic property C and finite decomposition complexity, arXiv:1607.00445

\bibitem{BN} G. Bell and A. Nag\'orko, On stability of asymptotic property C for products and some group
 extensions, arXiv:1607.05181

\bibitem{BDLM}  N.Brodskiy, J.Dydak, M.Levin, and A.Mitra, {\em    Hurewicz Theorem for Assouad-Nagata dimension}, Journal of the London Math.Soc. (2008) 77 (3): 741--756. 

 \bibitem{CDV1} M. Cencelj, J. Dydak, A. Vavpeti\v c, {\em Property A and asymptotic dimension}, Glasnik Matemati\v cki, Vol. 47(67)(2012), 441--444 (arXiv:0812.2619)

  \bibitem{CDV2} M. Cencelj, J. Dydak, A. Vavpeti\v c, {\em Asymptotic dimension, Property A, and Lipschitz maps}, Revista Matematica Complutense 26 (2013), pp. 561--571	(arXiv:0909.4095)


  \bibitem{CDV3} M. Cencelj, J. Dydak, A. Vavpeti\v c,  \emph{ Large scale versus small scale}, in Recent Progress in General Topology III, Hart, K.P.; van Mill, Jan; Simon, P (Eds.) (Atlantis Press, 2014), pp.165--204.
 
\bibitem{Dra1} A. N. Dranishnikov. Asymptotic topology. Uspekhi Mat. Nauk, 55(6(336)):71--116, 2000.

\bibitem{DZ} A. Dranishnikov and M. Zarichnyi. Asymptotic dimension, decomposition complexity, and Haver's property C. Topology Appl., 169:99--107, 2014.

\bibitem{Dyd1} J. Dydak. Coarse amenability and discreteness. J. Aust. Math. Soc., 100(1):65--77, 2016.

\bibitem{DV} J. Dydak and Z. Virk. Preserving coarse properties. Rev. Mat. Complut., 29(1):191--206, 2016.

\bibitem{GuentnerPermanence} Erik Guentner, Permanence in coarse geometry., Recent progress in general topology III. Based on the presentations at the Prague symposium, Prague, Czech Republic, 2001, Amsterdam: Atlantis Press, 2014, pp. 507--533 (English).


\bibitem{GHW} Erik Guentner, Nigel Higson, and Shmuel Weinberger. The Novikov conjecture for linear groups. Publ. Math. Inst. Hautes Etudes Sci., (101):243--268, 2005.

\bibitem{GTY1} Erik Guentner, Romain Tessera, and Guoliang Yu, A notion of geometric complexity and its application to topological rigidity, Invent. Math. 189 (2012), no. 2, 315--357. MR 2947546

\bibitem{GTY2} Erik Guentner, Romain Tessera, and Guoliang Yu, Discrete groups with finite decomposition complexity, Groups Geom. Dyn. 7 (2013), no. 2, 377--402. MR 3054574

\bibitem{Hol} M.Holloway, \textit{Duality of Scales}, PhD thesis, University of Tennessee 2016.

\bibitem{Kas} Daniel Kasprowski, On the K-theory of groups with finite decomposition complexity, Proceedings of the London Mathematical Society 110 (2015), no. 3, 565--592.


\bibitem{NR} Andrew Nicas and David Rosenthal, Hyperbolic dimension and decomposition complexity, arXiv:1509.06437.

\bibitem{NowakYu} P.W. Nowak and G. Yu, {\em Large scale geometry}, EMS Textbooks in Mathematics, European Mathematical Society (EMS), Zurich, 2012. 

\bibitem{O$_1$}
 P. Ostrand, {\em A conjecture of J. Nagata on dimension and
metrization}, Bull. Amer. Math. Soc. {\bf 71} (1965), 623--625.

\bibitem{OstrandDimofMetriSpacesHilbert} P.A. Ostrand, {\em Dimension of
Metric Spaces and Hilbert's problem 13}, Bull. Amer. Math. Soc.
{\bf 71} (1965), 619--622.

\bibitem{RR} Daniel A. Ramras and Bobby W. Ramsey, Extending Properties to Relatively Hyperbolic Groups, arXiv:1410.0060.

\bibitem{RTY} Daniel A. Ramras, Romain Tessera, Guoliang Yu, Finite decomposition complexity and the integral Novikov conjecture for higher algebraic K-theory, arXiv:1111.7022


\bibitem{RY1} A. Ranicki, M. Yamasaki. Controlled K-theory. Topology Appl. 61 (1995), no. 1, 1--59.

\bibitem{RY2} A. Ranicki, M. Yamasaki. Controlled L-theory. Exotic homology manifolds, Oberwolfach 2003,
Geom. Topol. Monogr. 9 (2006), 105--153.

\bibitem{Roe} John Roe, Lectures on coarse geometry, University Lecture Series, vol. 31, American Mathematical Society, Providence, RI, 2003. MR 2007488 (2004g:53050)

\bibitem{STG} G. Skandalis, J. L. Tu, and G. Yu. The coarse Baum-Connes conjecture and groupoids.
Topology, 41(4):807--834, 2002.

\bibitem{Yam} T. Yamauchi. Asymptotic property C of the countable direct sum of the integers. Topology
Appl., 184:50--53, 2015.

\bibitem{GYu} Guoliang Yu. The coarse Baum-Connes conjecture for spaces which admit a uniform embedding into Hilbert space. Invent. Math., 139(1):201--240, 2000.


\end{thebibliography}
\end{document}